\documentclass[11pt,oneside]{article}
\usepackage{supertabular, setspace,hyperref}
\usepackage[utf8]{inputenc}
\usepackage{amsmath}
\usepackage{amsfonts}
\usepackage{amssymb}
\usepackage{amsthm}
\usepackage{mathtools}
\newtheorem{theorem}{Theorem}[section]

\newtheorem{example}[theorem]{Example}

\newtheorem{remark}{\sc Remark}
\newtheorem{lemma}{\sc Lemma}[section]
\newtheorem{corollary}{\sc Corollary}[section]
\newtheorem{definition}{\sc Definition}[section]

\newcommand{\be}{\begin{eqnarray}}
\newcommand{\ee}{\end{eqnarray}}
\newcommand{\Be}{\begin{eqnarray*}}
	\newcommand{\Ee}{\end{eqnarray*}}
\newcommand{\bee}{\begin{equation}}
\newcommand{\eee}{\end{equation}}
\newcommand{\ba}{\begin{array}}
	\newcommand{\ea}{\end{array}}
\newcommand{\bl}{\begin{lemma}}
	\newcommand{\el}{\end{lemma}}
\newcommand{\bd}{\begin{definition}}
	\newcommand{\ed}{\end{definition}}
\newcommand{\bt}{\begin{theorem}}
	\newcommand{\et}{\end{theorem}}
\newcommand{\bp}{\begin{proof}}
	\newcommand{\ep}{\end{proof}}
\newcommand{\bi}{\begin{itemize}}
	\newcommand{\ei}{\end{itemize}}
\newcommand{\br}{\begin{remark}}
	\newcommand{\er}{\end{remark}}
\newcommand{\bc}{\begin{corollary}}
	\newcommand{\ec}{\end{corollary}}
\newcommand{\bex}{\begin{example}}
	\newcommand{\eex}{\end{example}}
\begin{document}
	\date{}
	\title{\textbf{On the flag curvature of homogeneous Finsler space with some special $(\alpha, \beta)$-metrics}}
	\maketitle
	\begin{center}
		\author{\textbf{Gauree Shanker and Kirandeep Kaur}}
	\end{center}
	\begin{center}
		Department of Mathematics and Statistics\\
		School of Basic and Applied Sciences\\
		Central University of Punjab, Bathinda, Punjab-151001, India\\
		Email:   gshankar@cup.ac.in, kiran5iitd@yahoo.com
	\end{center}
	\begin{center}
		\textbf{Abstract}
	\end{center}
	\begin{small}
			In this paper, first we derive an explicit formula for the flag curvature of a homogeneous Finsler space with infinite series $(\alpha, \beta)$-metric and exponential metric. Next, we deduce it  for naturally reductive homogeneous Finsler space with the above mentioned metrics. 
	\end{small}\\
	\textbf{2010 Mathematics Subject Classification:} 22E60, 53C30, 53C60.\\
	\textbf{Keywords and Phrases:} Homogeneous Finsler space, naturally reductive homogeneous Finsler space,  infinite series $(\alpha, \beta)$-metric, exponential metric, flag curvature, bi-invariant Riemannian metric.
	\section{Introduction}
	The main purpose of this paper is to give a formula for  flag curvture of a homogeneous Finsler space with infinite series $(\alpha, \beta)$-metric and also for exponential metric.  Finsler geometry is just the Riemannian geometry without the quadratic restriction, as mentioned by S. S. Chern (\cite{1996 Chern}). The notion of  $(\alpha, \beta)$-metric in Finsler geometry was first introduced by M. Matsumoto in 1972 (\cite{M.Mat1972}). An $(\alpha,\beta)$-metric is a Finsler metric of the form $F= \alpha \phi(s), \ s= \dfrac{\beta}{\alpha}$, where $\alpha= \sqrt{a_{ij}(x)y^iy^j}$ is induced by a Riemannian metric $ \tilde{a}=a_{ij}dx^i \otimes dx^j$ on a connected smooth $n$-manifold $M$ and $\beta= b_i(x) y^i$ is a 1-form on $M$. It is a well known fact that  $(\alpha,\beta)$-metrics are the generalizations of the Randers metric, introduced by G. Randers, in (\cite{Randers}). $(\alpha,\beta)$-metrics have various applications in physics and biology  (\cite{AIM}). 	Consider the $r^{\text{th}}$ series $(\alpha,\beta)$-metric:
	$$ F(\alpha, \beta)= \beta \sum_{r=0}^{r= \infty}\left(\dfrac{\alpha}{\beta} \right)^r .$$
	If $r=1$, then it is a Randers metric.\\
	If $r=\infty$, then $$ F= \dfrac{\beta^2}{\beta - \alpha}.  $$ 
	This metric is called an infinite series $(\alpha,\beta)$-metric. Interesting fact about this metric is that, it is the difference of a Randers metric and a Matsumoto metric, and satisfies Shen's lemma (see lemma \ref{existence of metric}).\\
	Some other important class of  $(\alpha,\beta)$-metrics are Randers metric, Kropina metric, Matsumoto metric and exponential metric. Various properties of $(\alpha,\beta)$-metrics have been studied by so many authors (\cite{M.Mat1992}, \cite{GS2011}, \cite{SB1}, \cite{SB2}, \cite{GKscur}, \cite{GSK},  \cite{SR}).  The study of various types of curvatures of Finsler spaces such as $S$-curvature, mean Berwald curvature, flag curvature always remain the central idea in the Finsler geometry.\\ 
	
	  Finsler geometry has been developing rapidly since last few decades, after its emergence in 1917 (\cite{Finsler}). Finsler geometry has been influenced by group theory. The celebrated Erlangen program of F. Klein, posed in 1872 (\cite{Klein}), greatly influenced the development of geometry. Later, Klein proposed to categorize the geometries by their chacteristic group of transformations. Famous Myers-Steenrod theorem, published in 1939 (\cite{Myers-Steenrod}), extended the scope of applying  Lie theory to all homogeneous Riemannian manifolds.\\
	
	\textbf{Theorem}({\bf Myers-Steenrod}): \textit{Let $ M $ be a connected Riemannian manifold. Then the group of isometries $ I(M) $ of $ M $ admits a differentiable structure such that $ I(M) $ is a Lie transformation group of $ M $}.\\
	
	 S. Deng and Z. Hou (\cite{DH})  have generalized this theorem to the Finslerian case, in 2002. This result opened the door for applying Lie theory to study Finsler geometry. The current important and interesting topics of research  in Finsler geometry are homogeneous Finsler spaces, Finsler spaces with $(\alpha, \beta)$-metrics, symmetric spaces, Rigidity problems and minimal Finsler surfaces etc. Almost all known examples of Einstein manifolds are homogeneous spaces. \\
	 
	 To compute the geometric quantities, specially, curvatures is an interesting problem in homogeneous spaces. In 1976, Milnor (\cite{Milnor}) studied the curvature properties of such spaces by using the formula for the sectional curvature of a left invariant Riemannian metric on a Lie group.\\ 
	 
	 In literature, the term curvature refers to the measurement of how fastly a curve is pulled away from the tangent line at some point in its domain. It should be noted that if a curve is oriented, then curvature remains invariant. 
	 The amount that a Riemannian manifold deviates from being Euclidean is defined by the term sectional (or Riemannian) curvature which is the generalization of Gaussian curvature of surfaces. Therefore, sectional curvature of Euclidean space is zero. The notion of sectional curvature was introduced by Riemann (\cite{Riemann}), in 1854.  At that time, Riemann had not given a method to calculate the sectional curvature,  it was done  by Christoffel (\cite{Christoffel}) in the year 1869.\\
	 
	  The flag curvature in Finsler geometry is a generalization of the sectional curvature of Riemannian geometry. The notion of flag curvature was introduced by Berwald in 1926 (\cite{Berwald}). Flag curvature is an important geometric quantity to characterize Finsler spaces. It is very difficult to compute the flag curvature of a general Finsler space as local coordinates are invoved in computation. It can be calculated for homogeneous Finsler spaces without using local coordinates. Many authors (\cite{Milnor}, \cite{N}, \cite{Puttmann}) have studied homogeneous Riemannian spaces. Some authors (\cite{DHIF}, \cite{Invariant}, \cite{EsraMoghQuotient},  \cite{Myers-Steenrod}) have tried to extend some of the results on homogeneous Riemannian spaces to homogeneous Finsler spaces. In the last decade, curvature properties of homogeneous Finsler spaces have been studied by some authors (\cite{DengHuflag2013}, \cite{EsraMoghflag2006}, \cite{MoghKro},  \cite{GKscur}). Classical Bonnet-Myers theorem, Cartan-Hadamard theorem and some other important results indicate that the sign of flag curvature governs the behaviour of geodesic flow and also it reflects the topology of the underlying manifold (\cite{Huang2017}). \\
	  
	  In (\cite{DHIF}), Deng and Hou  give an algebraic description of invariant Finsler metrics on reductive homogeneous spaces and find a necessary and sufficient condition for a homogeneous space to have invariant Finsler metrics. Also, in the same year, they study invariant Randers metrics on homogeneous manifolds and obtain a formula for the flag curvature of such spaces (\cite{Invariant}). But in 2006, with the help of some counter examples, Esrafilian and Salimi Moghaddam (\cite{EsraMoghflag2006})  prove that the formula for flag curvature of the Randers space, given in (\cite{Invariant})  is incorrect and they derive an explicit formula for the flag curvature of invariant Randers metrics on naturally reductive homogeneous manifolds. In (\cite{Moghflag2008}) and  (\cite{MoghflagR2008}), Moghaddam  give the formula for flag curvature of homogeneous spaces with square metric and Randers metric respectively. In 2014, Moghaddam (\cite{Mogh2014}) gives the formula for flag curvature of homogeneous space with Matsumoto metric. Recently, in 2016, Parhizkar and Latifi (\cite{ParLatiFlag}) have given the formula for  flag curvature of invariant $(\alpha, \beta)$-metrics on homogeneous Finsler spaces.
	
	\section{Preliminaries}
		In this section, we give some basic concepts of Finsler geometry that are required for next sections. For symbols and notations , we refer  (\cite{BCS}, \cite{CSBOOK})  and (\cite{Homogeneous Finsler Spaces}).
	\begin{definition}
		Let $M$ be a  smooth manifold of dimension $n, T_pM$  the tangent space at any point $p \in M.$ A real valued  bilinear function   
		$g\colon T_pM\times T_pM\longrightarrow [0,\infty)$
		is called  a Riemannian metric if it is symmetric and positive-definite,i.e., $\forall\;X,Y \in \mathfrak{X}(M), $ 
		\begin{enumerate}
			\item[\bf(i)]$ g(X,Y)=g(Y,X).$
			\item[\bf(ii)] $ g(X,X)\geq0$ and $ g(X,X)=0$ if and only if $X=0.$ 
		\end{enumerate}
		A smooth manifold with a given Riemannian metric is called a \textbf{Riemannian manifold.}
	\end{definition}
	\begin{definition}
		An n-dimensional real vector space $V$ is called a \textbf{Minkowski space}
		if there exists a real valued function $F:V \longrightarrow \mathbb{R}$ satisfying the following conditions: 
		\begin{enumerate}
			\item[\bf(a)]  $F$ is smooth on $V \backslash \{0\},$ 
			\item[\bf(b)] $F(v) \geq 0  \ \ \forall \ v \in V,$
			\item[\bf(c)] $F$ is positively homogeneous, i.e., $ F(\lambda v)= \lambda F(v), \ \ \forall \ \lambda > 0, $
			\item[\bf(d)] For a basis $\{v_1,\ v_2, \,..., \ v_n\}$ of $V$ and $y= y^iv_i \in V$, the Hessian matrix $\left( g_{_{ij}}\right)= \left( \dfrac{1}{2} F^2_{y^i y^j} \right)  $ is positive-definite at every point of $V \backslash \{0\}.$
		\end{enumerate} 
	Here,	$F$ is called a Minkowski norm.
	\end{definition}
	\begin{definition}	
		A connected smooth manifold $M$ is called a \textbf{Finsler space} if there exists a function $F\colon TM \longrightarrow [0, \infty)$ such that $F$ is smooth on the slit tangent bundle $TM \backslash \{0\}$ and the restriction of $F$ to any $T_p(M), \ p \in M$, is a Minkowski norm. In this case, $F$ is called a Finsler metric. 
	\end{definition}
	Let $(M, F)$ be a Finsler space and let $(x^i,y^i)$ be a standard coordinate system in $T_x(M)$. The induced inner product $g_y$ on $T_x(M)$ is given by $g_y(u,v)=g_{ij}(x,y)u^i v^j$, where $u=u^i \dfrac{\partial}{\partial x^i}, \ v=v^i\dfrac{\partial}{\partial x^i} \in T_x(M) $. Also note that $F(x,y)= \sqrt{g_y(y,y)}.$  \\
	The  condition for an $(\alpha,\beta)$-metric to be a Finsler metric is given in following Shen's lemma:
	\begin{lemma}(\cite{CSBOOK}) {\label{existence of metric}}
		Let $F=\alpha \phi(s), \ s=\beta/ \alpha,$ where $\alpha$ is a Riemannian metric and $\beta$ is a 1-form whose length with respect to $\alpha$ is bounded above, i.e., $b:=\lVert \beta \rVert_{\alpha} < b_0,$ where $b_0$ is a positive real number. Then $F$ is a Finsler metric if and only if the function $\phi=\phi(s)$ is a smooth positive function on $\left( -b_0, b_0\right) $ and satisfies the following condition:
		$$ \phi(s)-s\phi'(s)+\left( b^2-s^2\right) \phi''(s)>0, \ \ \lvert s\rvert \leq b < b_0.$$
		
	\end{lemma}
	
\begin{definition}
	Let $ (M, F) $ be a Finsler space. A diffeomorphism $\phi\colon  M \longrightarrow M$ is called an \textbf{isometry} if $F\left( \phi(p),  d\phi_p(X)\right) =F(p,X)$ for any $p \in M$ and $X \in T_p(M).$
\end{definition}

\begin{definition}
	Let $ G $ be a Lie group and $ M $ a smooth manifold. If $ G $ has a smooth action on $ M $, then $ G $ is called a Lie transformation group of $ M $.
\end{definition}
	\begin{definition}
		A connected Finsler space $ (M, F)$  is said to be homogeneous Finsler space if the action of the group of isometries of $(M,F)$, denoted by $ I(M, F)$, is transitive on $M$. 
	\end{definition}
\section{Flag Curvature} 
 Let $G$ be a compact Lie group and $H$ be a closed subgroup of $G$ with Lie algebras $\mathfrak{g}$ and $\mathfrak{h}$ respectively. Let $\tilde{a}_{_{0}}$ be a biinvariant Riemannian metric on $G$. The tangent space of homogeneous space $G/H$ is given by orthogonal complement $\mathfrak{m}$ of $\mathfrak{h}$ in $\mathfrak{g}$ with respect to $\tilde{a}_{_{0}}$. Each invariant metric $\tilde{a}$ on $G/H$ is determined by its restriction to $\mathfrak{m}$. The Ad($H$)-invariant inner product on $\mathfrak{m}$ arising  from $\tilde{a}$ can be extended to an   Ad($H$)-invariant inner product on $\mathfrak{g}$ by taking $\tilde{a}_{_{0}}$ for the components in $\mathfrak{h}$. In this way, $\tilde{a}$ determines a unique left invariant metric on $G$ which we also denote by $\tilde{a}$. Also, the values of $\tilde{a}_{_{0}}$ and $\tilde{a}$ at the identity $e$ of $G$ are inner products on $\mathfrak{g}$ and we denote them by $\left\langle \left\langle \ , \ \right\rangle \right\rangle $ and $\left\langle \  , \ \right\rangle $ respectively. The inner product $\left\langle \ , \ \right\rangle $ determines a positive definite endomorphism $\psi$ of $\mathfrak{g}$ such that $\left\langle Y , Z \right\rangle = \left\langle \left\langle \psi(Y) ,\  Z\right\rangle \right\rangle \ \ \forall \  Y,\ Z \in \mathfrak{g}. $\\
 
 P$\ddot{u}$ttmann (\cite{Puttmann}) has given the following formula for curvature tensor of the invariant metric $\left\langle \ , \ \right\rangle$ on the compact homogeneous space $G/H$:
 \begin{align*}
  \left\langle R(X, Y)Z, W\right\rangle &=\dfrac{1}{2}\bigg\{ \left\langle \left\langle B_{-}(X, Y), [Z, W]\right\rangle \right\rangle + \left\langle \left\langle [X, Y], B_{-}(Z,W) \right\rangle \right\rangle \bigg\} \\
  & \ \ \ \  + \dfrac{1}{4}\bigg\{ \left\langle [X, W], [Y, Z]_\mathfrak{m} \right\rangle -\left\langle [X, Z], [Y,W]_\mathfrak{m}\right\rangle -2 \left\langle [X, Y], [Z, W]_\mathfrak{m} \right\rangle \bigg\} \\
  & \ \ \ \ +\left\langle \left\langle B_{+}(X, W), \psi^{-1}B_{+}(Y, Z)\right\rangle \right\rangle - \left\langle \left\langle B_{+}(X, Z),\psi^{-1}B_{+}(Y, W)\right\rangle \right\rangle, 
 \end{align*}
 where the bilinear maps $B_{+}$ (symmetric) and $B_{-}$
(skew-symmetric) are defined by 
\begin{align*}
B_{+}(X,Y) &= \dfrac{1}{2}\bigg([X,\psi Y] + [Y, \psi X]\bigg),\\
B_{-}(X,Y) &= \dfrac{1}{2}\bigg([\psi X, Y] + [X, \psi Y]\bigg),
 \end{align*}
 and $[ \ , \ ]_{\mathfrak{m}}$ is the projection of $[ \ , \ ]$ to $\mathfrak{m}$.
 \begin{theorem}
 	
 	Let $\mathfrak{g}$ and $\mathfrak{h}$ be Lie algebras of the compact Lie group $G$ and its closed subgroup $H$ respectively. Further, let $\tilde{a}_{0} $ be a biinvariant metric on $G$ and $\tilde{a}$ be an invariant Riemannian metric on homogeneous space $G/H$ such that $ \left\langle Y, Z\right\rangle = \left\langle \left\langle \psi(Y), Z\right\rangle \right\rangle $, where $\psi \colon \mathfrak{g} \longrightarrow \mathfrak{g}$ is a positive definite endomorphism and $Y, Z \ \in \mathfrak{g}$. Also, suppose that $\tilde{X}$ is an invariant vector field on $G/H$ which is parallel with respect to $\tilde{a}$ and $\sqrt{\tilde{a}(\tilde{X}, \tilde{X})} < 1$ and $\tilde{X}_{H} = X.$ Assume that $F= \dfrac{\beta^2}{\beta-\alpha}$ be an infinite series $(\alpha, \beta)$-metric arising from $\tilde{a}$ and $\tilde{X}$ and $(P, Y)$ be a flag in $T_{H}(G/H)$ such that $\{ U, Y\}$ is an orthonormal basis of $P$ with respect to $\left\langle \ , \ \right\rangle $. Then the flag curvature of the flag $(P, Y)$ is given by 
 	
 	\begin{equation}{\label{flagcurinfeq}}
 	\begin{split}
 	K(P, Y) &= \left( \dfrac{\left\langle X, Y\right\rangle -1 }{\left\langle X, Y \right\rangle }\right)^4  \times \\
 	& \ \ \ \dfrac{\bigg[  \left\langle X, R(U, Y)Y \right\rangle \left\langle X, U\right\rangle \left\lbrace \left\langle X, Y\right\rangle^2 +2  \right\rbrace  +
 		\left\langle U, R(U, Y)Y \right\rangle \left\langle X, Y\right\rangle^2 \big\{ \left\langle X, Y\right\rangle -1 \big\}
 		\bigg] }{\bigg[ \left\langle X, Y\right\rangle^2 \big\{ \left\langle X, Y\right\rangle^3 +\left\langle X, Y\right\rangle^2 -5\left\langle X, Y \right\rangle + 3 \big \} + 2 \left\langle X, U\right\rangle ^2\big\{ \left\langle X, Y \right\rangle ^2  + 4 \left\langle X, Y \right\rangle -5 \bigg]   },
 	\end{split}
 	\end{equation}
 	where the values of $ \left\langle X, R(U, Y)Y \right\rangle $ and $ \left\langle U, R(U, Y)Y \right\rangle $ are given by equations (\ref{flaginfeq5}) and (\ref{flaginfeq7}). 
\end{theorem}
\begin{proof}
	Since $\tilde{X}$ is parallel with respect to $\tilde{a}$, so $\beta$ is parallel with respect
	to $\tilde{a}$. Therefore $F$ is a Berwarld metric, i.e., the Chern connection
	of F coincide with the Riemannian connection of $\tilde{a}$ . Thus the Finsler metric F has the same curvature tensor as that of  the Riemannian metric $\tilde{a}$ and  we denote it by $ R$. Using lemma 3.1 of (\cite{GKscur}), we can write 
	 $$ F\left( Y\right) =  \dfrac{\left\langle  X, Y\right\rangle^2}{\left\langle  X, Y\right\rangle - \sqrt{ \left\langle Y , Y\right\rangle } }. $$
	 Also, we know that
	 $$ g_{_{Y}}(U, V)= \dfrac{1}{2}\dfrac{\partial^2}{\partial s \partial t}F^2 (Y+sU + tV)\bigg|_{s=t=0}.  $$
	After  some calculations, we get
	\begin{equation}{\label{flaginfeq1}}
	\begin{split}
	 g_{_{Y}}(U, V)&= \dfrac{\left\langle X, Y\right\rangle^2}{\left( \left\langle  X, Y\right\rangle - \sqrt{ \left\langle Y , Y\right\rangle } \right) ^4 } \bigg[ \left\langle X, Y\right\rangle^2 \left\langle X, V\right\rangle \left\langle X, U \right\rangle 
	  -4 \left\langle Y, Y\right\rangle^ {3/2} \left\langle X, V \right\rangle \left\langle X, U\right\rangle \\
	  &\ \ \ \ \ \ + 6\left\langle Y, Y\right\rangle\left\langle X, V\right\rangle \left\langle X, U\right\rangle + \dfrac{\left\langle X, Y \right\rangle^2 \left\langle X, V \right\rangle \left\langle U, Y\right\rangle }{\sqrt{\left\langle Y, Y\right\rangle}} -4 \left\langle X, Y \right\rangle  \left\langle X, V\right\rangle \left\langle U, Y\right\rangle   \\
	  & \ \ \ \ \ \ - \dfrac{\left\langle X, Y\right\rangle^3 \left\langle U, Y \right\rangle \left\langle V, Y\right\rangle}{\left\langle Y, Y\right\rangle ^{3/2}} + \dfrac{\left\langle X, Y\right\rangle ^3 \left\langle U, V\right\rangle }{\sqrt{\left\langle Y, Y\right\rangle}} +  \dfrac{ 4 \left\langle X, Y\right\rangle ^2 \left\langle U, Y\right\rangle  \left\langle V, Y\right\rangle  }{\left\langle Y, Y\right\rangle } \\
	  & \ \ \ \ \ \ - \left\langle X, Y \right\rangle^2 \left\langle U, V\right\rangle + \dfrac{\left\langle X, Y \right\rangle^2 \left\langle X, U\right\rangle \left\langle V, Y \right\rangle}{\sqrt{\left\langle Y, Y\right\rangle}}  -4 \left\langle X, Y \right\rangle \left\langle X, U \right\rangle \left\langle V, Y \right\rangle \bigg]. 
	   \end{split}
	\end{equation}
	Since $ \{ U, Y\}$ is an orthonormal basis with respect to $ \left\langle \ , \ \right\rangle $, therefore  equation (\ref{flaginfeq1}) reduces to the following equation:
	\begin{equation}{\label{flaginfeq2}}
	\begin{split}
	g_{_{Y}}(U, V)&= \dfrac{\left\langle X, Y\right\rangle^2}{\left( \left\langle  X, Y\right\rangle - 1 \right) ^4 } \bigg[ \left\langle X, Y\right\rangle^2 \left\langle X, U\right\rangle \big\{\left\langle X, V\right\rangle + \left\langle V, Y\right\rangle  \big\}  + 2 \left\langle X, V\right\rangle \left\langle X, U \right\rangle \\
	& \ \ \ \ \ \ -4 \left\langle X, Y \right\rangle \left\langle X, U\right\rangle \left\langle V, Y\right\rangle  + \left\langle X, Y\right\rangle ^2 \left\langle U, V \right\rangle \big\{ \left\langle X, Y\right\rangle - 1\big\} \bigg].
	\end{split}
	\end{equation}
	From equation (\ref{flaginfeq2}), we  deduce following three equations:
	\begin{equation*}
	\begin{split}
	g_{_{Y}}(Y, Y) &= \dfrac{\left\langle X, Y\right\rangle^4}{\big( \left\langle  X, Y\right\rangle - 1 \big) ^4 } \bigg[ \left\langle X, Y\right\rangle^2 + 2 \left\langle X, Y\right\rangle -3 \bigg],\\
	g_{_{Y}}(Y, U) &= \dfrac{\left\langle X, Y\right\rangle^3}{\big( \left\langle  X, Y\right\rangle - 1 \big) ^4 } \bigg[ \left\langle X, Y\right\rangle^2 \left\langle X, U\right\rangle + \left\langle X, Y\right\rangle   \left\langle X, U \right\rangle -2 \left\langle X, U\right\rangle \bigg],\\
	g_{_{Y}}(U, U) &= \dfrac{\left\langle X, Y\right\rangle^2}{\big( \left\langle  X, Y\right\rangle - 1 \big) ^4 } \bigg[ \left\langle X , U\right\rangle ^2 \left\langle X, Y\right\rangle^2 + 2 \left\langle X, U\right\rangle^2 - \left\langle X, Y\right\rangle^2 + \left\langle X, Y\right\rangle^3 \bigg].   
	\end{split}
	\end{equation*} 
	Therefore,
	\begin{equation}{\label{flaginfeq3}}
	\begin{split}
	g_{_{Y}}(Y, Y) g_{_{Y}}(U, U)- g^2_{_{Y}}(Y, U) = \dfrac{\left\langle X, Y\right\rangle^6}{\big( \left\langle  X, Y\right\rangle - 1 \big) ^8 } \bigg[ \left\langle X, Y\right\rangle^5 +\left\langle X, Y\right\rangle^4 -5\left\langle X, Y \right\rangle^3 + 3\left\langle X, Y\right\rangle^2 \\
	 \ \ \ \ \ \ + 2 \left\langle X, Y \right\rangle ^2 \left\langle X, U\right\rangle ^2 + 8 \left\langle X, Y \right\rangle \left\langle X, U\right\rangle ^2 -10 \left\langle X, U\right\rangle ^2 \bigg].   
	 \end{split}  
	\end{equation}
	Also, 
	\begin{equation}{\label{flaginfeq4}}
	\begin{split}
	g_{_{Y}}(U, R(U, Y)Y)  = \dfrac{\left\langle X, Y\right\rangle^2}{\big( \left\langle  X, Y\right\rangle - 1 \big) ^4 } \bigg[ & \left\langle X, R(U, Y)Y \right\rangle \left\langle X, U\right\rangle \left\lbrace \left\langle X, Y\right\rangle^2 +2  \right\rbrace  \\
	& \left\langle Y, R(U, Y)Y \right\rangle \left\langle X, Y\right\rangle \left\langle X, U\right\rangle \Big\{ \left\langle X, Y\right\rangle -4 \Big\} \\
	& \left\langle U, R(U, Y)Y \right\rangle \left\langle X, Y\right\rangle^2 \Big\{ \left\langle X, Y\right\rangle -1 \Big\}
	  \bigg].   
	\end{split}  
	\end{equation}
   Using P$\ddot{u}$ttmann's formula, we have
  \begin{equation}{\label{flaginfeq5}}
  \begin{split}
  \left\langle X, R(U, Y)Y \right\rangle= & \dfrac{1}{4}\bigg( \left\langle \left\langle \left[ \psi U, Y \right] + \left[ U, \psi Y\right]  , \left[ Y, X\right]  \right\rangle \right\rangle + \left\langle \left\langle \left[ U, Y\right], \left[\psi Y, X \right] + \left[ Y, \psi X \right]  \right\rangle \right\rangle \bigg)\\
   &+ \dfrac{3}{4}\left\langle \left[ Y, U\right], \left[Y, X \right]_{\mathfrak{m} } \right\rangle \\
   	&+ \dfrac{1}{2}\left\langle \left\langle \left[ U, \psi X\right] + \left[ X, \psi U\right], \psi ^{-1} \left[ Y, \psi Y\right]   \right\rangle \right\rangle \\
   	&- \dfrac{1}{4} \left\langle \left\langle \left[ U, \psi Y \right]+ \left[ Y, \psi U\right], \psi ^{-1} \big( \left[ Y, \psi X\right]+ \left[ X, \psi Y\right]   \big)   \right\rangle \right\rangle, 
   	\end{split}
  \end{equation}
  
  \begin{equation}{\label{flaginfeq6}}
  \left\langle Y, R(U, Y)Y \right\rangle =0,
  \end{equation}
  \begin{equation}{\label{flaginfeq7}}
   \begin{split}
  \left\langle U, R(U, Y)Y \right\rangle= & \dfrac{1}{2}\bigg( \left\langle \left\langle \left[ \psi U, Y \right] + \left[ U, \psi Y\right]  , \left[ Y, U\right]  \right\rangle \right\rangle \bigg)\\
  &+ \dfrac{3}{4}\left\langle \left[ Y, U\right], \left[Y, U \right]_{\mathfrak{m} } \right\rangle \\
  &+ \left\langle \left\langle \left[ U, \psi U\right] , \psi ^{-1} \left[ Y, \psi Y\right]   \right\rangle \right\rangle \\
  &- \dfrac{1}{4} \left\langle \left\langle \left[ U, \psi Y \right]+ \left[ Y, \psi U\right], \psi ^{-1} \big( \left[ Y, \psi U\right]+ \left[ U, \psi Y\right]   \big)   \right\rangle \right\rangle. 
  \end{split}
  \end{equation}
  The flag curvature is given by
  $$ K(P, Y)= \dfrac{g_{_{Y}}(U, R(U, Y)Y)}{	g_{_{Y}}(Y, Y) g_{_{Y}}(U, U)- g^2_{_{Y}}(Y, U)}.$$
 Finally, using equations (\ref{flaginfeq3}) to  (\ref{flaginfeq7}) in the above formula, we get equation (\ref{flagcurinfeq}).
\end{proof}
Next, we derive formula for flag curvature of homogeneous Finsler space with exponential metric.
\begin{theorem}
Let $\mathfrak{g}$ and $\mathfrak{h}$ be Lie algebras of the compact Lie group $G$ and its closed subgroup $H$ respectively. Further, let $\tilde{a}_{0} $ be a biinvariant metric on $G$ and $\tilde{a}$ be an invariant Riemannian metric on homogeneous space $G/H$ such that $ \left\langle Y, Z\right\rangle = \left\langle \left\langle \psi(Y), Z\right\rangle \right\rangle $, where $\psi \colon \mathfrak{g} \longrightarrow \mathfrak{g}$ is a positive definite endomorphism and $Y, Z \ \in \mathfrak{g}$. Also, suppose that $\tilde{X}$ is an invariant vector field on $G/H$ which is parallel with respect to $\tilde{a}$ and $\sqrt{\tilde{a}(\tilde{X}, \tilde{X})} < 1$ and $\tilde{X}_{H} = X.$ Assume that $F= \alpha  e^{\beta/\alpha}$ be the exponential $(\alpha, \beta)$-metric arising from $\tilde{a}$ and $\tilde{X}$ and $(P, Y)$ be a flag in $T_{H}(G/H)$ such that $\{ U, Y\}$ is an orthonormal basis of $P$ with respect to $\left\langle \ , \ \right\rangle $. Then the flag curvature of the flag $(P, Y)$ is given by 
	\begin{equation}{\label{flagcurexpeq}}
	\begin{split}
	K(P, Y) = \dfrac{ \bigg[  2 \left\langle X, R(U, Y)Y\right\rangle \left\langle X, U \right\rangle + \left\langle U, R(U, Y)Y\right\rangle \Big( 1- \left\langle X, Y\right\rangle \Big) \bigg]}{e^{2\left\langle  X, Y\right\rangle}  \bigg[ \Big( 1- \left\langle X, Y\right\rangle \Big) \left( 1 + 4 \left\langle X, Y\right\rangle^2 \right)   + \left\langle X, U\right\rangle^2 \left( 1+ 8 \left\langle X, Y\right\rangle^2  \right)  \bigg]},
	\end{split}
	\end{equation}
	where the values of $ \left\langle X, R(U, Y)Y \right\rangle $ and $ \left\langle U, R(U, Y)Y \right\rangle $ are given by equations (\ref{flagexpeq5}) and (\ref{flagexpeq7}).
\end{theorem}

\begin{proof}
	Since $\tilde{X}$ is parallel with respect to $\tilde{a}$, so $\beta$ is parallel with respect
	to $\tilde{a}$. Therefore $F$ is a Berwarld metric, i.e., the Chern connection
	of F coincide with the Riemannian connection of $\tilde{a}$ . Thus the Finsler metric F has the same curvature tensor as that of  the Riemannian metric $\tilde{a}$ and  we denote it by $ R$. Using lemma 3.4 of (\cite{GKscur}), we can write 
$$	F\left( Y\right) = \left\langle  X, Y\right\rangle e^{\left\langle  X, Y\right\rangle / \sqrt{ \left\langle Y , Y\right\rangle }}. $$
	Also, we know that 
	$$g_{_{Y}}(U, V)= \dfrac{1}{2}\dfrac{\partial^2}{\partial s \partial t}F^2 (Y+sU + tV)\bigg|_{s=t=0}.$$
	After  some calculations, we get
	\begin{equation}{\label{flagexpeq1}}
	\begin{split}
	g_{_{Y}}(U, V)= e^{2\left\langle  X, Y\right\rangle / \sqrt{ \left\langle Y , Y\right\rangle }}  \Bigg[ & \left\langle U, V\right\rangle + 2 \left\langle X, U \right\rangle  \left\langle X, V\right\rangle - \dfrac{\left\langle X, Y\right\rangle \left\langle Y, U\right\rangle \left\langle Y, V\right\rangle   }{\left\langle Y, Y\right\rangle^{3/2} } \\
	&   + \dfrac{1}{\sqrt{ \left\langle Y , Y\right\rangle }}\bigg\{ \left\langle X, U \right\rangle \left\langle Y, V\right\rangle + \left\langle X, V \right\rangle \left\langle Y, U\right\rangle - \left\langle X, Y\right\rangle \left\langle U, V \right\rangle \bigg\} \\
	 &  + \dfrac{2\left\langle X, Y\right\rangle }{\left\langle Y, Y \right\rangle } \left\lbrace \dfrac{\left\langle X, Y\right\rangle \left\langle Y, U \right\rangle \left\langle Y, V\right\rangle}{\left\langle Y, Y\right\rangle } - \left\langle Y, U\right\rangle \left\langle X, V\right\rangle - \left\langle X, U \right\rangle \left\langle Y, V\right\rangle  \right\rbrace \Bigg].
	\end{split}
	\end{equation}
	Since $ \{ U, Y\}$ is an orthonormal basis with respect to $ \left\langle \ , \ \right\rangle $, therefore  equation (\ref{flagexpeq1}) reduces to the following equation:
	\begin{equation}{\label{flagexpeq2}}
	\begin{split}
	g_{_{Y}}(U, V)= e^{2\left\langle  X, Y\right\rangle}  \bigg[ &\left\langle U, V \right\rangle + 2 \left\langle X, U\right\rangle \left\langle X,V\right\rangle + \left\langle X, U \right\rangle \left\langle Y, V\right\rangle \\
	&   - \left\langle X, Y\right\rangle \left\langle U, V\right\rangle  - 2 \left\langle X, Y\right\rangle \left\langle X, U\right\rangle \left\langle Y, V \right\rangle \bigg].
	\end{split}
	\end{equation}
	From equation (\ref{flagexpeq2}), we  deduce following three equations:
	\begin{equation*}
	\begin{split}
	g_{_{Y}}(Y, Y) &= e^{2\left\langle  X, Y\right\rangle} \bigg[ 1 + 4 \left\langle X, Y \right\rangle ^2\bigg] ,\\
	g_{_{Y}}(Y, U) &= e^{2\left\langle  X, Y\right\rangle} \left\langle X, U \right\rangle  ,\\
	g_{_{Y}}(U, U) &= e^{2\left\langle  X, Y\right\rangle} \bigg[ 1 + 2 \left\langle X, U \right\rangle ^2 - \left\langle X, Y\right\rangle  \bigg].   
	\end{split}
	\end{equation*} 
	Therefore,
	\begin{equation}{\label{flagexpeq3}}
	\begin{split}
	g_{_{Y}}(Y, Y) g_{_{Y}}(U, U)- g^2_{_{Y}}(Y, U) = e^{4\left\langle  X, Y\right\rangle}  \bigg[ &\Big( 1- \left\langle X, Y\right\rangle \Big) \left( 1 + 4 \left\langle X, Y\right\rangle^2 \right) \\
	&  + \left\langle X, U\right\rangle^2 \left( 1+ 8 \left\langle X, Y\right\rangle^2  \right)  \bigg].
	\end{split}  
	\end{equation}
	Also, 
	\begin{equation}{\label{flagexpeq4}}
	\begin{split}
	g_{_{Y}}(U, R(U, Y)Y)  = e^{2\left\langle  X, Y\right\rangle} \bigg[ & 2 \left\langle X, R(U, Y)Y\right\rangle \left\langle X, U \right\rangle  \\
	& + \left\langle Y, R(U, Y)Y\right\rangle\left\langle X, U \right\rangle \Big( 1-2 \left\langle X, Y\right\rangle \Big) \\
	&  + \left\langle U, R(U, Y)Y\right\rangle \Big( 1- \left\langle X, Y\right\rangle \Big) \bigg].
	\end{split}  
	\end{equation}
	Using P$\ddot{u}$ttmann's formula, we have
	\begin{equation}{\label{flagexpeq5}}
	\begin{split}
	\left\langle X, R(U, Y)Y \right\rangle= & \dfrac{1}{4}\bigg( \left\langle \left\langle \left[ \psi U, Y \right] + \left[ U, \psi Y\right]  , \left[ Y, X\right]  \right\rangle \right\rangle + \left\langle \left\langle \left[ U, Y\right], \left[\psi Y, X \right] + \left[ Y, \psi X \right]  \right\rangle \right\rangle \bigg)\\
	&+ \dfrac{3}{4}\left\langle \left[ Y, U\right], \left[Y, X \right]_{\mathfrak{m} } \right\rangle \\
	&+ \dfrac{1}{2}\left\langle \left\langle \left[ U, \psi X\right] + \left[ X, \psi U\right], \psi ^{-1} \left[ Y, \psi Y\right]   \right\rangle \right\rangle \\
	&- \dfrac{1}{4} \left\langle \left\langle \left[ U, \psi Y \right]+ \left[ Y, \psi U\right], \psi ^{-1} \big( \left[ Y, \psi X\right]+ \left[ X, \psi Y\right]   \big)   \right\rangle \right\rangle, 
	\end{split}
	\end{equation}
	
	\begin{equation}{\label{flagexpeq6}}
	\left\langle Y, R(U, Y)Y \right\rangle =0,
	\end{equation}
	\begin{equation}{\label{flagexpeq7}}
	\begin{split}
	\left\langle U, R(U, Y)Y \right\rangle= & \dfrac{1}{2}\bigg( \left\langle \left\langle \left[ \psi U, Y \right] + \left[ U, \psi Y\right]  , \left[ Y, U\right]  \right\rangle \right\rangle \bigg)\\
	&+ \dfrac{3}{4}\left\langle \left[ Y, U\right], \left[Y, U \right]_{\mathfrak{m} } \right\rangle \\
	&+ \left\langle \left\langle \left[ U, \psi U\right] , \psi ^{-1} \left[ Y, \psi Y\right]   \right\rangle \right\rangle \\
	&- \dfrac{1}{4} \left\langle \left\langle \left[ U, \psi Y \right]+ \left[ Y, \psi U\right], \psi ^{-1} \big( \left[ Y, \psi U\right]+ \left[ U, \psi Y\right]   \big)   \right\rangle \right\rangle, 
	\end{split}
	\end{equation}
	The flag curvature is given by
	$$ K(P, Y)= \dfrac{g_{_{Y}}(U, R(U, Y)Y)}{	g_{_{Y}}(Y, Y) g_{_{Y}}(U, U)- g^2_{_{Y}}(Y, U)}.$$
	Finally, using equations (\ref{flagexpeq3}) to  (\ref{flagexpeq7}) in the above formula, we get equation (\ref{flagcurexpeq}).
\end{proof}
\section{Naturally reductive homogeneous space}
In the literature, there are two different definitions of naturally reductive homogeneous Finsler spaces. In 2004 ,  Deng and Hou (\cite{DHIF}) have given the first definition of naturally reductive homogeneous Finsler spaces. In 2007, Latifi (\cite{Latifi2007geo}) has given another definition  of naturally reductive homogeneous Finsler spaces. In 2010, Deng and Hou (\cite{DengHou2010natred}) have shown that if a homogeneous Finsler space is naturally reductive in the sense of Latifi, then it must be naturally reductive in the sense of Deng and Hou and also is Berwaldian. Parhizkar and Moghaddam (\cite{ParMoghnatred}) prove that both the definitions of naturally reductive homogeneous Finsler spaces are equivalent under the consideration of a mild condition.\\
Next, recall(\cite{KN}, \cite{DHIF}, \cite{Latifi2007geo}) the following definitions for later use.
\begin{definition}
	A homogeneous space $G/H$ with an invariant Riemannian metric $\tilde{a}$ is called naturally reductive if there exists an \text{Ad}$(H)$-invariant decomposition $\mathfrak{g}=\mathfrak{m}+\mathfrak{h}$ such that
	$$ \left\langle  \left[ X, Y\right]_\mathfrak{m} , Z\right\rangle + \left\langle Y, \left[ X, Z\right]_\mathfrak{m}  \right\rangle =0, \ \forall \ X, Y, Z \in \mathfrak{m},  $$ where $\left\langle \ , \ \right\rangle $ is the bilinear form on $\mathfrak{m}$ induced by $\tilde{a}.$ \\
	In particular, if we consider $H= \{e\},$ then $\tilde{a}$ is bi-invariant Riemannian metric on $G$ and the above condition reduces to the following:
	
	$$ \left\langle  \left[ X, Y\right] , Z\right\rangle + \left\langle Y, \left[ X, Z\right]  \right\rangle =0, \ \forall \ X, Y, Z \in \mathfrak{g}.  $$
\end{definition}

\begin{definition} 
	A homogeneous Finsler space $G/H$ with an invariant Finsler metric $F$ is said to be naturally reductive if there exists an invariant Riemannian metric $\tilde{a}$ on $G/H$ such that $\left( G/H, \tilde{a}\right) $ is naturally reductive and the Chern connection of $F$ coincides with the Levi-Civita connection of $\tilde{a}.$ 
\end{definition}
\begin{definition}
	A homogeneous Finsler space $G/H$ with an invariant Finsler metric $F$ is called naturally reductive if there exists an \text{Ad}$(H)$-invariant decomposition $\mathfrak{g}= \mathfrak{m}+ \mathfrak{h}$ such that
	$$ g_{_{Y}}\big( \left[ X, U\right]_{\mathfrak{m}}, V \big)\  + \ g_{_{Y}}\big( U, \left[ X, V \right]_{\mathfrak{m}} \big)\  +\  2\  C_{_{Y}} \big( \left[ X, Y \right]_{\mathfrak{m}}, U, V \big) = 0,$$
	where $\ Y \neq 0, \ X, U, V \in \mathfrak{m}.$
\end{definition}
\begin{theorem}
	Let $\left( G/H , F\right) $ be a homogeneous Finsler space of Berwald type	with infinite series metric $F=\dfrac{\beta^2}{\beta-\alpha}$ defined by an invariant Riemannian metric $\left\langle \ , \ \right\rangle $ and an invariant vector field $\tilde{X}$ such that $\tilde{X}(H)= X.$ Then $\left( G/H , F\right) $ is naturally reductive if and only if the underlying Riemannian space $\left( G/H , \left\langle \ , \ \right\rangle  \right)$ is naturally reductive.
\end{theorem}
\begin{proof} Let $Y \neq 0,\  Z \in \mathfrak{m}.$
	From equation (\ref{flaginfeq1}), we can write
	\begin{equation}{\label{natinfeq2}}
	\begin{split}
	g_{_{Y}}(Y, \left[ Y, Z\right]_{\mathfrak{m}} )&= \dfrac{\left\langle X, Y\right\rangle^2}{\left( \left\langle  X, Y\right\rangle - \sqrt{ \left\langle Y , Y\right\rangle } \right) ^4 } \bigg[ \left\langle X, Y\right\rangle^2 \left\langle X, \left[ Y, Z\right]_{\mathfrak{m}}\right\rangle \left\langle X, Y \right\rangle \\
	& \ \ \ -4 \left\langle Y, Y\right\rangle^ {3/2} \left\langle X, \left[ Y, Z\right]_{\mathfrak{m}} \right\rangle \left\langle X, Y\right\rangle 
	+ 6\left\langle Y, Y\right\rangle\left\langle X, \left[ Y, Z\right]_{\mathfrak{m}}\right\rangle \left\langle X, Y\right\rangle\\
	& \ \ \  + \dfrac{\left\langle X, Y \right\rangle^2 \left\langle X, \left[ Y, Z\right]_{\mathfrak{m}} \right\rangle \left\langle Y, Y\right\rangle }{\sqrt{\left\langle Y, Y\right\rangle}} -4 \left\langle X, Y \right\rangle  \left\langle X, \left[ Y, Z\right]_{\mathfrak{m}}\right\rangle \left\langle Y, Y\right\rangle   \\
	& \ \ \  - \dfrac{\left\langle X, Y\right\rangle^3 \left\langle Y, Y_{\mathfrak{m}} \right\rangle \left\langle \left[ Y, Z\right]_{\mathfrak{m}}, Y\right\rangle}{\left\langle Y, Y\right\rangle ^{3/2}} + \dfrac{\left\langle X, Y\right\rangle ^3 \left\langle Y, \left[ Y, Z\right]_{\mathfrak{m}}\right\rangle }{\sqrt{\left\langle Y, Y\right\rangle}}\\
	& \ \ \  +  \dfrac{ 4 \left\langle X, Y\right\rangle ^2 \left\langle Y, Y\right\rangle  \left\langle \left[ Y, Z\right]_{\mathfrak{m}}, Y\right\rangle  }{\left\langle Y, Y\right\rangle }  - \left\langle X,Y \right\rangle^2 \left\langle Y, \left[ Y, Z\right]_{\mathfrak{m}}\right\rangle \\
	& \ \ \ + \dfrac{\left\langle X, Y\right\rangle^2 \left\langle X, Y\right\rangle \left\langle \left[ Y, Z\right]_{\mathfrak{m}}, Y \right\rangle}{\sqrt{\left\langle Y, Y\right\rangle}}  -4 \left\langle X, Y \right\rangle \left\langle X, Y \right\rangle \left\langle \left[ Y, Z\right]_{\mathfrak{m}}, Y \right\rangle \bigg]. 
	\\&
	= \dfrac{\left\langle X, Y\right\rangle^3}{\left( \left\langle  X, Y\right\rangle - \sqrt{ \left\langle Y , Y\right\rangle } \right) ^4 } \times \\
	&\ \ \ \ \Bigg[ \left\langle X, \left[ Y, Z\right]_{\mathfrak{m}}\right\rangle \Big\{\left\langle X, Y\right\rangle^2 -4 \left\langle Y, Y\right\rangle ^{3/2} + \left\langle X, Y\right\rangle \sqrt{\left\langle Y, Y\right\rangle} + 2 \left\langle Y, Y \right\rangle \Big\}\\
	 &\ \ \ \  +\left\langle Y, \left[ Y, Z\right]_{\mathfrak{m}}\right\rangle \left\lbrace  \dfrac{\left\langle X, Y\right\rangle^2 }{\sqrt{\left\langle Y, Y \right\rangle }} - \left\langle X, Y\right\rangle \right\rbrace \Bigg].
	\end{split}
	\end{equation}
	Since $(G/H, F)$ is of Berwald type, therefore Chern connection of $\left( G/H , F\right) $ coincides with the Riemannian connection of  $\left( G/H , \left\langle \ , \ \right\rangle  \right)$.\\
	Thus from equation (\ref{natinfeq2}), we have
	\begin{equation}{\label{natinfeq3}}
	\left\langle X, \left[ Y, Z\right]_{\mathfrak{m}}\right\rangle =0, \ \forall \ Z \in \mathfrak{m}.
	\end{equation}
	Now,	let $\left( G/H , F\right) $ be naturally reductive, i.e.,  in the sense of (\cite{Latifi2007geo}), we have
	$$ g_{_{Y}}\big( \left[ Z, U\right]_{\mathfrak{m}}, V \big) \  + \ g_{_{Y}}\big( \left[ Z, V\right]_{\mathfrak{m}}, U \big) \  + \ 2 \  C_{_{Y}}\big( \left[ Z, Y\right]_{\mathfrak{m}}, U, V\big) =0, \ \forall \ Y \neq 0,\  U,\  V, \ Z \in \mathfrak{m}.$$
	Therefore, we can write
	$$ g_{_{Y}}\big( \left[ Y, U\right]_{\mathfrak{m}}, V \big) \  + \ g_{_{Y}}\big( \left[ Y, V\right]_{\mathfrak{m}}, U \big) \  + \ 2 \  C_{_{Y}}\big( \left[ Y, Y\right]_{\mathfrak{m}}, U, V\big) =0, \ \forall \ Y \neq 0,$$
	i.e.,
	\begin{equation}{\label{natinfeq4}}
	g_{_{Y}}\big( \left[ Y, U\right]_{\mathfrak{m}}, V \big) \  + \ g_{_{Y}}\big( \left[ Y, V\right]_{\mathfrak{m}}, U \big)=0.
	\end{equation}
	From above equation, we have
	\begin{equation}{\label{natinfeq5}}
	g_{_{Y}}\big( \left[ Y,  Z\right]_{\mathfrak{m}}, Y \big) = 0.
	\end{equation}
	From equation (\ref{natinfeq2}), (\ref{natinfeq3}) and (\ref{natinfeq5}), we get
	\begin{equation}{\label{natinfeq6}}
	\left\langle \left[ Y,  Z\right]_{\mathfrak{m}}, Y \right\rangle = 0.
	\end{equation}
	From equation (\ref{flaginfeq1}), we can write
	\begin{equation*}
	\begin{split}
	g_{_{Y}}(\left[ Y, U\right]_{\mathfrak{m}} , V)&= \dfrac{\left\langle X, Y\right\rangle^2}{\left( \left\langle  X, Y\right\rangle - \sqrt{ \left\langle Y , Y\right\rangle } \right) ^4 } \bigg[  \left\langle X, Y\right\rangle^2 \left\langle X, V\right\rangle \left\langle X, \left[ Y, U\right]_{\mathfrak{m}} \right\rangle \\
	& \ \ \ \ \ \  -4 \left\langle Y, Y\right\rangle^ {3/2} \left\langle X, V \right\rangle \left\langle X, \left[ Y, U\right]_{\mathfrak{m}}\right\rangle 
	+ 6\left\langle Y, Y\right\rangle\left\langle X, V\right\rangle \left\langle X, \left[ Y, U\right]_{\mathfrak{m}}\right\rangle \\
	& \ \ \ \ \ \ + \dfrac{\left\langle X, Y \right\rangle^2 \left\langle X, V \right\rangle \left\langle \left[ Y, U\right]_{\mathfrak{m}}, Y\right\rangle }{\sqrt{\left\langle Y, Y\right\rangle}} 
	-4 \left\langle X, Y \right\rangle  \left\langle X, V\right\rangle \left\langle \left[ Y, U\right]_{\mathfrak{m}}, Y\right\rangle \\
	& \ \ \ \ \ \ - \dfrac{\left\langle X, Y\right\rangle^3 \left\langle \left[ Y, U\right]_{\mathfrak{m}}, Y \right\rangle \left\langle V, Y\right\rangle}{\left\langle Y, Y\right\rangle ^{3/2}} + \dfrac{\left\langle X, Y\right\rangle ^3 \left\langle\left[ Y, U\right]_{\mathfrak{m}}, V\right\rangle }{\sqrt{\left\langle Y, Y\right\rangle}} \\
	& \ \ \ \ \ \ +  \dfrac{ 4 \left\langle X, Y\right\rangle ^2 \left\langle \left[ Y, U\right]_{\mathfrak{m}}, Y\right\rangle  \left\langle V, Y\right\rangle  }{\left\langle Y, Y\right\rangle } 
	- \left\langle X, Y \right\rangle^2 \left\langle \left[ Y, U\right]_{\mathfrak{m}}, V\right\rangle \\
	& \ \ \ \ \ \ + \dfrac{\left\langle X, Y \right\rangle^2 \left\langle X, \left[ Y, U\right]_{\mathfrak{m}}\right\rangle \left\langle V, Y \right\rangle}{\sqrt{\left\langle Y, Y\right\rangle}}  -4 \left\langle X, Y \right\rangle \left\langle X, \left[ Y, U\right]_{\mathfrak{m}} \right\rangle \left\langle V, Y \right\rangle \bigg]. 
	\end{split}
	\end{equation*}
	Using equations (\ref{natinfeq3}) and (\ref{natinfeq6}) in above equation, we get
	
	$$ g_{_{Y}}(\left[ Y, U\right]_{\mathfrak{m}} , V) =  \dfrac{\left\langle X, Y\right\rangle^4 \left\langle \left[ Y, U\right]_{\mathfrak{m}}, V\right\rangle }{\left( \left\langle  X, Y\right\rangle - \sqrt{ \left\langle Y , Y\right\rangle } \right) ^3  \sqrt{\left\langle Y, Y\right\rangle }}.  $$
	Similarly, we get
	$$ g_{_{Y}}(\left[ Y, V\right]_{\mathfrak{m}} , U) =  \dfrac{\left\langle X, Y\right\rangle^4 \left\langle \left[ Y, V\right]_{\mathfrak{m}}, U\right\rangle }{\left( \left\langle  X, Y\right\rangle - \sqrt{ \left\langle Y , Y\right\rangle } \right) ^3  \sqrt{\left\langle Y, Y\right\rangle }}.  $$
	Adding the above two equations, we get
	\begin{equation}{\label{natinfeq7}}
	g_{_{Y}}(\left[ Y, U\right]_{\mathfrak{m}} , V) +  g_{_{Y}}(\left[ Y, V\right]_{\mathfrak{m}} , U) = \dfrac{\left\langle X, Y\right\rangle^4  }{\left( \left\langle  X, Y\right\rangle - \sqrt{ \left\langle Y , Y\right\rangle } \right) ^3  \sqrt{\left\langle Y, Y\right\rangle }} \Big\{ \left\langle \left[ Y, U\right]_{\mathfrak{m}}, V\right\rangle + \left\langle \left[ Y, V\right]_{\mathfrak{m}}, U\right\rangle \Big\} .
	\end{equation}
	From equations (\ref{natinfeq4}) and (\ref{natinfeq7}), we get
	$$ \left\langle \left[ Y, U\right]_{\mathfrak{m}}, V\right\rangle + \left\langle \left[ Y, V\right]_{\mathfrak{m}}, U\right\rangle =0.$$
	Hence, $ (G/H, \left\langle \ , \ \right\rangle )$ is naturally reductive.\\
	
	Conversely, let $ (G/H, \left\langle \ , \ \right\rangle )$ be naturally reductive.
	Using equation (\ref{flaginfeq1}), we can write
	\begin{equation}{\label{natinfeq8}}
	\begin{split}
	g_{_{Y}}(\left[ Z, U\right]_{\mathfrak{m}} , V) = & \dfrac{\left\langle X, Y\right\rangle^3}{\left( \left\langle  X, Y\right\rangle - \sqrt{ \left\langle Y , Y\right\rangle } \right) ^4}  \Bigg[ \left\langle \left[ Z, U\right]_{\mathfrak{m}}, V\right\rangle \left\lbrace \dfrac{\left\langle X, Y\right\rangle^2 }{\sqrt{\left\langle Y, Y\right\rangle }} - \left\langle X, Y\right\rangle \right\rbrace \\
	& \hspace{2.5cm} + \left\langle \left[ Z, U\right]_{\mathfrak{m}}, Y\right\rangle \Bigg\{ \dfrac{\left\langle X, Y\right\rangle \left\langle X, V\right\rangle }{\sqrt{\left\langle Y, Y\right\rangle }} - 4 \left\langle X, V\right\rangle \\
	& \hspace{3cm} - \dfrac{\left\langle X, Y\right\rangle^2 \left\langle V, Y\right\rangle  }{\left\langle Y, Y\right\rangle^{3/2}} + \dfrac{4 \left\langle X, Y\right\rangle \left\langle V, Y\right\rangle}{\left\langle Y, Y\right\rangle } \Bigg\} \Bigg].
	\end{split}
	\end{equation}
	Similarly, we can write
	\begin{equation}{\label{natinfeq9}}
	\begin{split}
	g_{_{Y}}(\left[ Z, V\right]_{\mathfrak{m}} , U) = & \dfrac{\left\langle X, Y\right\rangle^3}{\left( \left\langle  X, Y\right\rangle - \sqrt{ \left\langle Y , Y\right\rangle } \right) ^4}  \Bigg[ \left\langle \left[ Z, V\right]_{\mathfrak{m}}, U\right\rangle \left\lbrace \dfrac{\left\langle X, Y\right\rangle^2 }{\sqrt{\left\langle Y, Y\right\rangle }} - \left\langle X, Y\right\rangle \right\rbrace \\
	& \hspace{2.5cm} + \left\langle \left[ Z, V\right]_{\mathfrak{m}}, Y\right\rangle \Bigg\{ \dfrac{\left\langle X, Y\right\rangle \left\langle X, U\right\rangle }{\sqrt{\left\langle Y, Y\right\rangle }} - 4 \left\langle X, U\right\rangle \\
	& \hspace{3cm} - \dfrac{\left\langle X, Y\right\rangle^2 \left\langle U, Y\right\rangle  }{\left\langle Y, Y\right\rangle^{3/2}} + \dfrac{4 \left\langle X, Y\right\rangle \left\langle U, Y\right\rangle}{\left\langle Y, Y\right\rangle } \Bigg\} \Bigg].
	\end{split}
	\end{equation}
	Next, Cartan tensor is given by
	$$ C_{_{Y}}\left( Z, U, V\right)= \dfrac{1}{2}\dfrac{d}{dt}\Big[ g_{_{Y + t V}}\left( Z, U \right)\Big] \bigg|_{t=0}. $$
	After computations, we get
	\begin{equation}{\label{natinfeq10}}
	\begin{split}
	2 C_{_{Y}}\left( \left[ Z, Y\right]_{\mathfrak{m}} , U, V\right) &=   \dfrac{\left\langle X, Y\right\rangle^3}{\left( \left\langle  X, Y\right\rangle - \sqrt{ \left\langle Y , Y\right\rangle } \right) ^4} \times \\
	& \Bigg[ \left\langle \left[ Z, Y\right]_{\mathfrak{m}}, U \right\rangle \Bigg\{ \dfrac{\left\langle X, Y\right\rangle \left\langle X, V\right\rangle }{\sqrt{\left\langle Y, Y\right\rangle }} - 4 \left\langle X, V\right\rangle  - \dfrac{\left\langle X, Y\right\rangle^2 \left\langle V, Y\right\rangle  }{\left\langle Y, Y\right\rangle^{3/2}} + \dfrac{4 \left\langle X, Y\right\rangle \left\langle V, Y\right\rangle}{\left\langle Y, Y\right\rangle } \Bigg\}\\
	&+ \left\langle \left[ Z, Y\right]_{\mathfrak{m}}, V\right\rangle \Bigg\{ \dfrac{\left\langle X, Y\right\rangle \left\langle X, U\right\rangle }{\sqrt{\left\langle Y, Y\right\rangle }} - 4 \left\langle X, U\right\rangle  - \dfrac{\left\langle X, Y\right\rangle^2 \left\langle U, Y\right\rangle  }{\left\langle Y, Y\right\rangle^{3/2}} + \dfrac{4 \left\langle X, Y\right\rangle \left\langle U, Y\right\rangle}{\left\langle Y, Y\right\rangle } \Bigg\} \Bigg].
	\end{split}
	\end{equation}
	Adding equations (\ref{natinfeq8}), (\ref{natinfeq9}) and (\ref{natinfeq10}), we get 
	\begin{equation*}
	\begin{split}
	& g_{_{Y}}(\left[ Z, U\right]_{\mathfrak{m}} , V) + g_{_{Y}}(\left[ Z, V\right]_{\mathfrak{m}} , U) + 2 C_{_{Y}}\left( \left[ Z, Y\right]_{\mathfrak{m}} , U, V\right)\\
	  & = \dfrac{\left\langle X, Y\right\rangle^3}{\left( \left\langle  X, Y\right\rangle - \sqrt{ \left\langle Y , Y\right\rangle } \right) ^4} \times \\
	& \Bigg[ \Bigg\{\left\langle \left[ Z, U\right]_{\mathfrak{m}}, V\right\rangle + \left\langle \left[ Z, V\right]_{\mathfrak{m}}, U\right\rangle \Bigg\} \Bigg\{\dfrac{\left\langle X, Y\right\rangle^2 }{\sqrt{\left\langle Y, Y\right\rangle }} - \left\langle X, Y\right\rangle  \Bigg\} \\
	& + \Bigg\{\left\langle \left[ Z, U\right]_{\mathfrak{m}}, Y\right\rangle + \left\langle \left[ Z, Y\right]_{\mathfrak{m}}, U\right\rangle \Bigg\} \Bigg\{ \dfrac{\left\langle X, Y\right\rangle \left\langle X, V\right\rangle }{\sqrt{\left\langle Y, Y\right\rangle }} - 4 \left\langle X, V\right\rangle  - \dfrac{\left\langle X, Y\right\rangle^2 \left\langle V, Y\right\rangle  }{\left\langle Y, Y\right\rangle^{3/2}} + \dfrac{4 \left\langle X, Y\right\rangle \left\langle V, Y\right\rangle}{\left\langle Y, Y\right\rangle } \Bigg\}\\
	& + \Bigg\{\left\langle \left[ Z, V\right]_{\mathfrak{m}}, Y\right\rangle + \left\langle \left[ Z, Y\right]_{\mathfrak{m}}, V\right\rangle \Bigg\} \Bigg\{ \dfrac{\left\langle X, Y\right\rangle \left\langle X, U\right\rangle }{\sqrt{\left\langle Y, Y\right\rangle }} - 4 \left\langle X, U\right\rangle  - \dfrac{\left\langle X, Y\right\rangle^2 \left\langle U, Y\right\rangle  }{\left\langle Y, Y\right\rangle^{3/2}} + \dfrac{4 \left\langle X, Y\right\rangle \left\langle U, Y\right\rangle}{\left\langle Y, Y\right\rangle } \Bigg\}\Bigg].\\
	&= \dfrac{\left\langle X, Y\right\rangle^3}{\left( \left\langle  X, Y\right\rangle - \sqrt{ \left\langle Y , Y\right\rangle } \right) ^4} \Big[ 0 +0 +0\Big], \text{because} \  (G/H, \left\langle \ , \ \right\rangle ) \ \text{is naturally reductive}.
	\end{split}
\end{equation*}
Hence, $\left( G/H , F\right) $ is naturally reductive.
\end{proof}
Next, we prove the similar result for homogeneous Finsler space with exponential metric.

\begin{theorem}
	Let $\left( G/H , F\right) $ be a homogeneous Finsler space of Berwald type	with exponential metric $F= \alpha  e^{\beta/\alpha}$ defined by an invariant Riemannian metric $\left\langle \ , \ \right\rangle $ and an invariant vector field $\tilde{X}$ such that $\tilde{X}(H)= X.$ Then $\left( G/H , F\right) $ is naturally reductive if and only if the underlying Riemannian space $\left( G/H , \left\langle \ , \ \right\rangle  \right)$ is naturally reductive.
\end{theorem}
\begin{proof} Let $Y \neq 0,\  Z \in \mathfrak{m}.$
	From equation (\ref{flagexpeq1}), we can write
	\begin{equation}{\label{natexpeq2}}
	\begin{split}
	g_{_{Y}}(Y, \left[ Y, Z\right]_{\mathfrak{m}} )=& e^{2\left\langle  X, Y\right\rangle / \sqrt{ \left\langle Y , Y\right\rangle }}  \Bigg[  \left\langle Y, \left[ Y, Z\right]_{\mathfrak{m}}\right\rangle + 2 \left\langle X, Y \right\rangle  \left\langle X, \left[ Y, Z\right]_{\mathfrak{m}}\right\rangle - \dfrac{\left\langle X, Y\right\rangle \left\langle Y, Y\right\rangle \left\langle Y, \left[ Y, Z\right]_{\mathfrak{m}}\right\rangle   }{\left\langle Y, Y\right\rangle^{3/2} } \\
	&   + \dfrac{1}{\sqrt{ \left\langle Y , Y\right\rangle }}\bigg\{ \left\langle X, Y \right\rangle \left\langle Y, \left[ Y, Z\right]_{\mathfrak{m}}\right\rangle + \left\langle X, \left[ Y, Z\right]_{\mathfrak{m}} \right\rangle \left\langle Y, Y\right\rangle - \left\langle X, Y\right\rangle \left\langle Y, \left[ Y, Z\right]_{\mathfrak{m}} \right\rangle \bigg\} \\
	&  + \dfrac{2\left\langle X, Y\right\rangle }{\left\langle Y, Y \right\rangle } \left\lbrace \dfrac{\left\langle X, Y\right\rangle \left\langle Y, Y \right\rangle \left\langle Y, \left[ Y, Z\right]_{\mathfrak{m}}\right\rangle}{\left\langle Y, Y\right\rangle } - \left\langle Y, Y\right\rangle \left\langle X, \left[ Y, Z\right]_{\mathfrak{m}}\right\rangle - \left\langle X, Y\right\rangle \left\langle Y, \left[ Y, Z\right]_{\mathfrak{m}}\right\rangle  \right\rbrace \Bigg]\\
	= &   e^{2\left\langle  X, Y\right\rangle / \sqrt{ \left\langle Y , Y\right\rangle }}  \Bigg[ \left\langle Y, \left[ Y, Z\right]_{\mathfrak{m}}\right\rangle \left\lbrace 1 - \dfrac{\left\langle X, Y\right\rangle }{\sqrt{\left\langle Y, Y \right\rangle }}\right\rbrace + \left\langle X, \left[ Y, Z\right]_{\mathfrak{m}} \right\rangle \sqrt{\left\langle Y, Y\right\rangle } \Bigg].
	\end{split}
	\end{equation}
	Since $(G/H, F)$ is of Berwald type, therefore Chern connection of $\left( G/H , F\right) $ coincides with the Riemannian connection of  $\left( G/H , \left\langle \ , \ \right\rangle  \right)$.\\
	Thus from equation (\ref{natexpeq2}), we have
	\begin{equation}{\label{natexpeq3}}
	\left\langle X, \left[ Y, Z\right]_{\mathfrak{m}}\right\rangle =0, \ \forall \ Z \in \mathfrak{m}.
	\end{equation}
	Now,	let $\left( G/H , F\right) $ be naturally reductive, i.e.,  in the sense of (\cite{Latifi2007geo}), we have
	$$ g_{_{Y}}\big( \left[ Z, U\right]_{\mathfrak{m}}, V \big) \  + \ g_{_{Y}}\big( \left[ Z, V\right]_{\mathfrak{m}}, U \big) \  + \ 2 \  C_{_{Y}}\big( \left[ Z, Y\right]_{\mathfrak{m}}, U, V\big) =0, \ \forall \ Y \neq 0,\  U,\  V, \ Z \in \mathfrak{m}.$$
	Therefore, we can write
	$$ g_{_{Y}}\big( \left[ Y, U\right]_{\mathfrak{m}}, V \big) \  + \ g_{_{Y}}\big( \left[ Y, V\right]_{\mathfrak{m}}, U \big) \  + \ 2 \  C_{_{Y}}\big( \left[ Y, Y\right]_{\mathfrak{m}}, U, V\big) =0, \ \forall \ Y \neq 0,$$
	i.e.,
	\begin{equation}{\label{natexpeq4}}
	g_{_{Y}}\big( \left[ Y, U\right]_{\mathfrak{m}}, V \big) \  + \ g_{_{Y}}\big( \left[ Y, V\right]_{\mathfrak{m}}, U \big)=0.
	\end{equation}
	From above equation, we have
	\begin{equation}{\label{natexpeq5}}
	g_{_{Y}}\big( \left[ Y,  Z\right]_{\mathfrak{m}}, Y \big) = 0.
	\end{equation}
	From equation (\ref{natexpeq2}), (\ref{natexpeq3}) and (\ref{natexpeq5}), we get
	\begin{equation}{\label{natexpeq6}}
	\left\langle \left[ Y,  Z\right]_{\mathfrak{m}}, Y \right\rangle = 0.
	\end{equation}
	From equation (\ref{flagexpeq1}), we can write
	\begin{equation*}
	\begin{split}
	g_{_{Y}}(\left[ Y, U\right]_{\mathfrak{m}} , V)=& e^{2\left\langle  X, Y\right\rangle / \sqrt{ \left\langle Y , Y\right\rangle }}  \Bigg[  \left\langle \left[ Y, U\right]_{\mathfrak{m}}, V\right\rangle + 2 \left\langle X, \left[ Y, U\right]_{\mathfrak{m}} \right\rangle  \left\langle X, V\right\rangle - \dfrac{\left\langle X, Y\right\rangle \left\langle Y, \left[ Y, U\right]_{\mathfrak{m}}\right\rangle \left\langle Y, V\right\rangle   }{\left\langle Y, Y\right\rangle^{3/2} } \\
	&   + \dfrac{1}{\sqrt{ \left\langle Y , Y\right\rangle }}\bigg\{ \left\langle X, \left[ Y, U\right]_{\mathfrak{m}} \right\rangle \left\langle Y, V\right\rangle + \left\langle X, V \right\rangle \left\langle Y, \left[ Y, U\right]_{\mathfrak{m}}\right\rangle - \left\langle X, Y\right\rangle \left\langle \left[ Y, U\right]_{\mathfrak{m}}, V \right\rangle \bigg\} \\
	&  + \dfrac{2\left\langle X, Y\right\rangle }{\left\langle Y, Y \right\rangle } \left\lbrace \dfrac{\left\langle X, Y\right\rangle \left\langle Y, \left[ Y, U\right]_{\mathfrak{m}} \right\rangle \left\langle Y, V\right\rangle}{\left\langle Y, Y\right\rangle } - \left\langle Y, \left[ Y, U\right]_{\mathfrak{m}}\right\rangle \left\langle X, V\right\rangle - \left\langle X, \left[ Y, U\right]_{\mathfrak{m}} \right\rangle \left\langle Y, V\right\rangle  \right\rbrace \Bigg].
	\end{split}
	\end{equation*}
	Using equations (\ref{natexpeq3}) and (\ref{natexpeq6}) in above equation, we get
	
	$$ g_{_{Y}}(\left[ Y, U\right]_{\mathfrak{m}} , V) =  e^{2\left\langle  X, Y\right\rangle / \sqrt{ \left\langle Y , Y\right\rangle }} \left\lbrace 1 - \dfrac{\left\langle X, Y\right\rangle }{\sqrt{\left\langle Y, Y \right\rangle }} \right\rbrace \left\langle \left[ Y, U\right]_{\mathfrak{m}} , V \right\rangle  .  $$
	Similarly, we get
	$$ g_{_{Y}}(\left[ Y, V\right]_{\mathfrak{m}} , U) = e^{2\left\langle  X, Y\right\rangle / \sqrt{ \left\langle Y , Y\right\rangle }} \left\lbrace 1 - \dfrac{\left\langle X, Y\right\rangle }{\sqrt{\left\langle Y, Y \right\rangle }} \right\rbrace \left\langle \left[ Y, V\right]_{\mathfrak{m}} , U \right\rangle   .  $$
	Adding the above two equations, we get
	\begin{equation}{\label{natexpeq7}}
	g_{_{Y}}(\left[ Y, U\right]_{\mathfrak{m}} , V) +  g_{_{Y}}(\left[ Y, V\right]_{\mathfrak{m}} , U) =  e^{2\left\langle  X, Y\right\rangle / \sqrt{ \left\langle Y , Y\right\rangle }} \left\lbrace 1 - \dfrac{\left\langle X, Y\right\rangle }{\sqrt{\left\langle Y, Y \right\rangle }} \right\rbrace \Big\{ \left\langle \left[ Y, U\right]_{\mathfrak{m}}, V\right\rangle + \left\langle \left[ Y, V\right]_{\mathfrak{m}}, U\right\rangle \Big\} .
	\end{equation}
	From equations (\ref{natexpeq4}) and (\ref{natexpeq7}), we get
	$$ \left\langle \left[ Y, U\right]_{\mathfrak{m}}, V\right\rangle + \left\langle \left[ Y, V\right]_{\mathfrak{m}}, U\right\rangle =0.$$
	Hence, $ (G/H, \left\langle \ , \ \right\rangle )$ is naturally reductive.\\
	
	Conversely, let $ (G/H, \left\langle \ , \ \right\rangle )$ be naturally reductive.
	Using equation (\ref{flaginfeq1}), we can write
	\begin{equation}{\label{natexpeq8}}
	\begin{split}
	g_{_{Y}}(\left[ Z, U\right]_{\mathfrak{m}}, V)= & e^{2\left\langle  X, Y\right\rangle / \sqrt{ \left\langle Y , Y\right\rangle }}  \Bigg[  \left\langle \left[ Z, U\right]_{\mathfrak{m}}, V\right\rangle + 2 \left\langle X, \left[ Z, U\right]_{\mathfrak{m}} \right\rangle  \left\langle X, V\right\rangle - \dfrac{\left\langle X, Y\right\rangle \left\langle Y, \left[ Z, U\right]_{\mathfrak{m}}\right\rangle \left\langle Y, V\right\rangle   }{\left\langle Y, Y\right\rangle^{3/2} } \\
	&   + \dfrac{1}{\sqrt{ \left\langle Y , Y\right\rangle }}\bigg\{ \left\langle X, \left[ Z, U\right]_{\mathfrak{m}} \right\rangle \left\langle Y, V\right\rangle + \left\langle X, V \right\rangle \left\langle Y, \left[ Z, U\right]_{\mathfrak{m}}\right\rangle - \left\langle X, Y\right\rangle \left\langle \left[ Z, U\right]_{\mathfrak{m}}, V \right\rangle \bigg\} \\
	&  + \dfrac{2\left\langle X, Y\right\rangle }{\left\langle Y, Y \right\rangle } \left\lbrace \dfrac{\left\langle X, Y\right\rangle \left\langle Y, \left[ Z, U\right]_{\mathfrak{m}} \right\rangle \left\langle Y, V\right\rangle}{\left\langle Y, Y\right\rangle } - \left\langle Y, \left[ Z, U\right]_{\mathfrak{m}}\right\rangle \left\langle X, V\right\rangle - \left\langle X, \left[ Z, U\right]_{\mathfrak{m}} \right\rangle \left\langle Y, V\right\rangle  \right\rbrace \Bigg]\\
	= &  e^{2\left\langle  X, Y\right\rangle / \sqrt{ \left\langle Y , Y\right\rangle }}  \Bigg[  \left\langle \left[ Z, U\right]_{\mathfrak{m}}, V\right\rangle \left\lbrace 1 - \dfrac{\left\langle X, Y\right\rangle }{\sqrt{\left\langle Y, Y \right\rangle }}\right\rbrace  
	 + \left\langle \left[ Z, U\right]_{\mathfrak{m}}, Y\right\rangle \\
	 & \ \ \ \ \ \  \left\lbrace \dfrac{- \left\langle X, Y\right\rangle \left\langle Y, V\right\rangle}{\left\langle Y, Y\right\rangle^{3/2} } + \dfrac{\left\langle X, V \right\rangle }{\sqrt{\left\langle Y, Y\right\rangle }} + \dfrac{2 \left\langle X, Y \right\rangle^2 \left\langle Y, V\right\rangle}{\left\langle Y, Y\right\rangle^2  } - \dfrac{2\left\langle X, Y\right\rangle \left\langle X, V\right\rangle  }{\left\langle Y, Y\right\rangle }\right\rbrace \Bigg].
	\end{split}
	\end{equation}
	Similarly, we can write
	\begin{equation}{\label{natexpeq9}}
	\begin{split}
	g_{_{Y}}(\left[ Z, V\right]_{\mathfrak{m}} , U) = &  e^{2\left\langle  X, Y\right\rangle / \sqrt{ \left\langle Y , Y\right\rangle }}  \Bigg[  \left\langle \left[ Z, V\right]_{\mathfrak{m}}, U\right\rangle \left\lbrace 1 - \dfrac{\left\langle X, Y\right\rangle }{\sqrt{\left\langle Y, Y \right\rangle }}\right\rbrace  
	+ \left\langle \left[ Z, V\right]_{\mathfrak{m}}, Y\right\rangle \\
	& \ \ \ \ \ \  \left\lbrace \dfrac{- \left\langle X, Y\right\rangle \left\langle Y, U\right\rangle}{\left\langle Y, Y\right\rangle^{3/2} } + \dfrac{\left\langle X, U \right\rangle }{\sqrt{\left\langle Y, Y\right\rangle }} + \dfrac{2 \left\langle X, Y \right\rangle^2 \left\langle Y, U\right\rangle}{\left\langle Y, Y\right\rangle^2  } - \dfrac{2\left\langle X, Y\right\rangle \left\langle X, U\right\rangle  }{\left\langle Y, Y\right\rangle }\right\rbrace \Bigg].
	\end{split}
	\end{equation}
	Next, Cartan tensor is given by
	$$ C_{_{Y}}\left( Z, U, V\right)= \dfrac{1}{2}\dfrac{d}{dt}\Big[ g_{_{Y + t V}}\left( Z, U \right)\Big] \bigg|_{t=0}. $$
	After computations, we get
	\begin{equation}{\label{natexpeq10}}
	\begin{split}
	2  C_{_{Y}} & \left( \left[ Z, Y\right]_{\mathfrak{m}} , U,  V\right) = e^{2\left\langle  X, Y\right\rangle / \sqrt{ \left\langle Y , Y\right\rangle }} \times \\
	& \Bigg[ \left\langle \left[ Z, Y\right]_{\mathfrak{m}}, V\right\rangle 
	  \left\lbrace \dfrac{- \left\langle X, Y\right\rangle \left\langle Y, U\right\rangle}{\left\langle Y, Y\right\rangle^{3/2} } + \dfrac{\left\langle X, U \right\rangle }{\sqrt{\left\langle Y, Y\right\rangle }} + \dfrac{2 \left\langle X, Y \right\rangle^2 \left\langle Y, U\right\rangle}{\left\langle Y, Y\right\rangle^2  } - \dfrac{2\left\langle X, Y\right\rangle \left\langle X, U\right\rangle  }{\left\langle Y, Y\right\rangle }\right\rbrace \\
	& + \left\langle \left[ Z, Y\right]_{\mathfrak{m}}, U\right\rangle 
	  \left\lbrace \dfrac{- \left\langle X, Y\right\rangle \left\langle Y, V\right\rangle}{\left\langle Y, Y\right\rangle^{3/2} } + \dfrac{\left\langle X, V \right\rangle }{\sqrt{\left\langle Y, Y\right\rangle }} + \dfrac{2 \left\langle X, Y \right\rangle^2 \left\langle Y, V\right\rangle}{\left\langle Y, Y\right\rangle^2  } - \dfrac{2\left\langle X, Y\right\rangle \left\langle X, V\right\rangle  }{\left\langle Y, Y\right\rangle }\right\rbrace\Bigg]  .
	\end{split}
	\end{equation}
	Adding equations (\ref{natexpeq8}), (\ref{natexpeq9}) and (\ref{natexpeq10}), we get 
	\begin{equation*}
	\begin{split}
	& g_{_{Y}}(\left[ Z, U\right]_{\mathfrak{m}} , V) + g_{_{Y}}(\left[ Z, V\right]_{\mathfrak{m}} , U) + 2 C_{_{Y}}\left( \left[ Z, Y\right]_{\mathfrak{m}} , U, V\right)\\
	& =  e^{2\left\langle  X, Y\right\rangle / \sqrt{ \left\langle Y , Y\right\rangle }} \times \\
	& \Bigg[ \Bigg\{\left\langle \left[ Z, U\right]_{\mathfrak{m}}, V\right\rangle + \left\langle \left[ Z, V\right]_{\mathfrak{m}}, U\right\rangle \Bigg\} \Bigg\{1 - \dfrac{\left\langle X, Y\right\rangle }{\sqrt{\left\langle Y, Y \right\rangle }} \Bigg\} \\
	& + \Bigg\{\left\langle \left[ Z, U\right]_{\mathfrak{m}}, Y\right\rangle + \left\langle \left[ Z, Y\right]_{\mathfrak{m}}, U\right\rangle \Bigg\} \Bigg\{ \dfrac{- \left\langle X, Y\right\rangle \left\langle Y, V\right\rangle}{\left\langle Y, Y\right\rangle^{3/2} } + \dfrac{\left\langle X, V \right\rangle }{\sqrt{\left\langle Y, Y\right\rangle }} + \dfrac{2 \left\langle X, Y \right\rangle^2 \left\langle Y, V\right\rangle}{\left\langle Y, Y\right\rangle^2  } - \dfrac{2\left\langle X, Y\right\rangle \left\langle X, V\right\rangle  }{\left\langle Y, Y\right\rangle } \Bigg\}\\
	& + \Bigg\{\left\langle \left[ Z, V\right]_{\mathfrak{m}}, Y\right\rangle + \left\langle \left[ Z, Y\right]_{\mathfrak{m}}, V\right\rangle \Bigg\} \Bigg\{ \dfrac{- \left\langle X, Y\right\rangle \left\langle Y, U\right\rangle}{\left\langle Y, Y\right\rangle^{3/2} } + \dfrac{\left\langle X, U \right\rangle }{\sqrt{\left\langle Y, Y\right\rangle }} + \dfrac{2 \left\langle X, Y \right\rangle^2 \left\langle Y, U\right\rangle}{\left\langle Y, Y\right\rangle^2  } - \dfrac{2\left\langle X, Y\right\rangle \left\langle X, U\right\rangle  }{\left\langle Y, Y\right\rangle } \Bigg\}\Bigg].\\
	&= e^{2\left\langle  X, Y\right\rangle / \sqrt{ \left\langle Y , Y\right\rangle }} \Big[ 0 +0 +0\Big], \text{because} \  (G/H, \left\langle \ , \ \right\rangle ) \ \text{is naturally reductive}.
	\end{split}
	\end{equation*}
	Hence, $\left( G/H , F\right) $ is naturally reductive.
	
	\end{proof}

\section{Flag curvature of naturally reductive homogeneous space}
 In   (\cite{ParMoghnatred}),  authors find a formula for flag curvature of naturally reductive homogeneous $(\alpha, \beta)$-metric spaces in the sense of Deng and Hou.

Here, we derive  formula for flag curvature of naturally reductive homogeneous Finsler space with infinite series $(\alpha, \beta)$-metric in the sense of Deng and Hou. 
\begin{theorem}
		Let $G/H$  be naturally reductive homogeneous Finsler space with infinite series $(\alpha, \beta)$-metric $F= \dfrac{\beta^2}{\beta-\alpha} $, defined by an invariant Riemannian metric $\tilde{a}$ and  an invariant vector field $\tilde{X}$ on $G/H$, Assume that  $\tilde{X}_{H} = X.$ and  $(P, Y)$ be a flag in $T_{H}(G/H)$ such that $\{ U, Y\}$ is an orthonormal basis of $P$ with respect to $\left\langle \ , \ \right\rangle $. Then the flag curvature of the flag $(P, Y)$ is given by 
\begin{equation}{\label{flaginfeq8}}
\begin{split}
K(P, Y) &= \left( \dfrac{\left\langle X, Y\right\rangle -1 }{\left\langle X, Y \right\rangle }\right)^4  \times \\
&  \dfrac{\splitdfrac{\bigg[  \left\lbrace \dfrac{1}{4}\left\langle X, \big[ Y, \left[ U, Y\right]_{\mathfrak{m}} \big]_{\mathfrak{m}}\right\rangle  + \left\langle X, \left[ Y, \left[ U, Y\right]_{\mathfrak{h}} \right] \right\rangle \right\rbrace  \left\langle X, U\right\rangle \left\lbrace \left\langle X, Y\right\rangle^2 +2  \right\rbrace } {  +
		\left\lbrace \dfrac{1}{4}\left\langle U, \big[ Y, \left[ U, Y\right]_{\mathfrak{m}} \big]_{\mathfrak{m}}\right\rangle  + \left\langle U
		, \left[ Y, \left[ U, Y\right]_{\mathfrak{h}} \right] \right\rangle \right\rbrace \left\langle X, Y\right\rangle^2 \big\{ \left\langle X, Y\right\rangle -1 \big\}
		\bigg] }}{\bigg[ \left\langle X, Y\right\rangle^2 \big\{ \left\langle X, Y\right\rangle^3 +\left\langle X, Y\right\rangle^2 -5\left\langle X, Y \right\rangle + 3 \big \} + 2 \left\langle X, U\right\rangle ^2\big\{ \left\langle X, Y \right\rangle ^2  + 4 \left\langle X, Y \right\rangle -5 \bigg]   }.
\end{split}
\end{equation}	
\end{theorem}
\begin{proof}
Since $F$ is naturally reductive, using proposition 3.4 in (\cite{KN}) on page 202, we have 
	$$ R(U, Y)Y= \dfrac{1}{4}\big[ Y, \left[ U, Y\right]_{\mathfrak{m}} \big]_{\mathfrak{m}} + \left[ Y, \left[ U, Y\right]_{\mathfrak{h}} \right], \ \forall \  U, Y \in \mathfrak{m}.$$
	Substituting above value in equation (\ref{flagcurinfeq}), we get 
	\begin{equation*}
	\begin{split}
	K(P, Y) &= \left( \dfrac{\left\langle X, Y\right\rangle -1 }{\left\langle X, Y \right\rangle }\right)^4  \times \\
	& \ \ \ \dfrac{\splitdfrac{\bigg[  \left\langle X, \dfrac{1}{4}\big[ Y, \left[ U, Y\right]_{\mathfrak{m}} \big]_{\mathfrak{m}} + \left[ Y, \left[ U, Y\right]_{\mathfrak{h}} \right] \right\rangle \left\langle X, U\right\rangle \left\lbrace \left\langle X, Y\right\rangle^2 +2  \right\rbrace } {  +
		\left\langle U, \dfrac{1}{4}\big[ Y, \left[ U, Y\right]_{\mathfrak{m}} \big]_{\mathfrak{m}} + \left[ Y, \left[ U, Y\right]_{\mathfrak{h}} \right] \right\rangle \left\langle X, Y\right\rangle^2 \big\{ \left\langle X, Y\right\rangle -1 \big\}
		\bigg] }}{\bigg[ \left\langle X, Y\right\rangle^2 \big\{ \left\langle X, Y\right\rangle^3 +\left\langle X, Y\right\rangle^2 -5\left\langle X, Y \right\rangle + 3 \big \} + 2 \left\langle X, U\right\rangle ^2\big\{ \left\langle X, Y \right\rangle ^2  + 4 \left\langle X, Y \right\rangle -5 \bigg]   }.
	\end{split}
	\end{equation*}
	Simplifying the above equation,we get equation (\ref{flaginfeq8}).
\end{proof}
If  $ H= \{e\}$, then we have the following corollary:
\begin{corollary}
Let $F = \dfrac{\beta^2}{\beta-\alpha}$ be defined by a bi-invariant Riemannian metric $\tilde{a}$  on a Lie group $G$ and  a left invariant vector field $X$  on $G$ such that the Chern connection of $F$ coincides with the Riemannian connection of $\tilde{a}$.   Assume that 	
	 $(P, Y)$ be a flag in $T_{e}(G)$ such that $\{ U, Y\}$ is an orthonormal basis of $P$ with respect to $\left\langle \ , \ \right\rangle $. Then the flag curvature of the flag $(P, Y)$ is given by
	 \begin{equation}{\label{flaginfeq9}}
	 \begin{split}
	 K(P, Y) &= \left( \dfrac{\left\langle X, Y\right\rangle -1 }{\left\langle X, Y \right\rangle }\right)^4  \times \\
	 & \ \ \ \dfrac{\bigg[  \left\langle X, \big[ Y, \left[ U, Y\right] \big]  \right\rangle \left\langle X, U\right\rangle \left\lbrace \left\langle X, Y\right\rangle^2 +2  \right\rbrace    +
	 	\left\langle U, \big[ Y, \left[ U, Y\right] \big] \right\rangle \left\langle X, Y\right\rangle^2 \big\{ \left\langle X, Y\right\rangle -1 \big\}
	 	\bigg] }{4\bigg[ \left\langle X, Y\right\rangle^2 \big\{ \left\langle X, Y\right\rangle^3 +\left\langle X, Y\right\rangle^2 -5\left\langle X, Y \right\rangle + 3 \big \} + 2 \left\langle X, U\right\rangle ^2\big\{ \left\langle X, Y \right\rangle ^2  + 4 \left\langle X, Y \right\rangle -5 \bigg]   }.
	 \end{split}
	 \end{equation}
\end{corollary}
\begin{proof}
	
	Since $ \left\langle \ , \ \right\rangle $ is bi-invariant, we have 
	$$ R(U, Y)Y= \dfrac{1}{4}\big[ Y, \left[ U, Y\right] \big].$$
		Substituting above value in equation (\ref{flagcurinfeq}), we get 
	\begin{equation*}
	\begin{split}
	K(P, Y) &= \left( \dfrac{\left\langle X, Y\right\rangle -1 }{\left\langle X, Y \right\rangle }\right)^4  \times \\
	& \ \ \ \dfrac{\bigg[  \left\langle X, \dfrac{1}{4}\big[ Y, \left[ U, Y\right] \big]  \right\rangle \left\langle X, U\right\rangle \left\lbrace \left\langle X, Y\right\rangle^2 +2  \right\rbrace    +
			\left\langle U, \dfrac{1}{4}\big[ Y, \left[ U, Y\right] \big] \right\rangle \left\langle X, Y\right\rangle^2 \big\{ \left\langle X, Y\right\rangle -1 \big\}
			\bigg] }{\bigg[ \left\langle X, Y\right\rangle^2 \big\{ \left\langle X, Y\right\rangle^3 +\left\langle X, Y\right\rangle^2 -5\left\langle X, Y \right\rangle + 3 \big \} + 2 \left\langle X, U\right\rangle ^2\big\{ \left\langle X, Y \right\rangle ^2  + 4 \left\langle X, Y \right\rangle -5 \bigg]   }.
	\end{split}
	\end{equation*}
	Simplifying the above equation, we get equation (\ref{flaginfeq9}).
\end{proof}

Next, we derive  formula for flag curvature of naturally reductive homogeneous Finsler space with exponential $(\alpha, \beta)$-metric in the sense of Deng and Hou. 
\begin{theorem}
	Let $G/H$  be naturally reductive homogeneous Finsler space with exponential $(\alpha, \beta)$-metric $F= \alpha  e^{\beta/\alpha}  $, defined by an invariant Riemannian metric $\tilde{a}$ and  an invariant vector field $\tilde{X}$ on $G/H$, Assume that  $\tilde{X}_{H} = X.$ and  $(P, Y)$ be a flag in $T_{H}(G/H)$ such that $\{ U, Y\}$ is an orthonormal basis of $P$ with respect to $\left\langle \ , \ \right\rangle $. Then the flag curvature of the flag $(P, Y)$ is given by 
	\begin{equation}{\label{flagexpeq8}}
		K(P, Y) = \dfrac{\splitdfrac{\Bigg[ \left\lbrace \dfrac{1}{2}\left\langle X, \big[ Y, \left[ U, Y\right]_{\mathfrak{m}} \big]_{\mathfrak{m}}\right\rangle  + 2 \left\langle X, \left[ Y, \left[ U, Y\right]_{\mathfrak{h}} \right] \right\rangle \right\rbrace  \left\langle X, U\right\rangle} {  +
			\left\lbrace \dfrac{1}{4}\left\langle U, \big[ Y, \left[ U, Y\right]_{\mathfrak{m}} \big]_{\mathfrak{m}}\right\rangle  + \left\langle U
			, \left[ Y, \left[ U, Y\right]_{\mathfrak{h}} \right] \right\rangle \right\rbrace \Big( 1- \left\langle X, Y\right\rangle \Big)	\Bigg] }}{e^{2\left\langle  X, Y\right\rangle}  \bigg[ \Big( 1- \left\langle X, Y\right\rangle \Big) \left( 1 + 4 \left\langle X, Y\right\rangle^2 \right)   + \left\langle X, U\right\rangle^2 \left( 1+ 8 \left\langle X, Y\right\rangle^2  \right)  \bigg]  }.
		\end{equation}	
\end{theorem}
\begin{proof}
	Since $F$ is naturally reductive, using proposition 3.4 in (\cite{KN}) on page 202, we have 
	$$ R(U, Y)Y= \dfrac{1}{4}\big[ Y, \left[ U, Y\right]_{\mathfrak{m}} \big]_{\mathfrak{m}} + \left[ Y, \left[ U, Y\right]_{\mathfrak{h}} \right], \ \forall \  U, Y \in \mathfrak{m}.$$
	Substituting above value in equation (\ref{flagcurexpeq}), we get 
	\begin{equation*}
		K(P, Y) = \dfrac{ \splitdfrac{\Bigg[  2 \left\langle X, \dfrac{1}{4}\big[ Y, \left[ U, Y\right]_{\mathfrak{m}} \big]_{\mathfrak{m}} + \left[ Y, \left[ U, Y\right]_{\mathfrak{h}} \right]\right\rangle \left\langle X, U \right\rangle}{ + \left\langle U, \dfrac{1}{4}\big[ Y, \left[ U, Y\right]_{\mathfrak{m}} \big]_{\mathfrak{m}} + \left[ Y, \left[ U, Y\right]_{\mathfrak{h}} \right]\right\rangle \Big( 1- \left\langle X, Y\right\rangle \Big) \Bigg]}}{e^{2\left\langle  X, Y\right\rangle}  \bigg[ \Big( 1- \left\langle X, Y\right\rangle \Big) \left( 1 + 4 \left\langle X, Y\right\rangle^2 \right)   + \left\langle X, U\right\rangle^2 \left( 1+ 8 \left\langle X, Y\right\rangle^2  \right)  \bigg]}.
		\end{equation*}
	Simplifying the above equation, we get equation (\ref{flagexpeq8}).
\end{proof}
If  $ H= \{e\}$, then we have the following corollary:
\begin{corollary}
	Let $F=\alpha  e^{\beta/\alpha} $ be defined by a bi-invariant Riemannian metric $\tilde{a}$  on a Lie group $G$ and  a left invariant vector field $X$  on $G$ such that the Chern connection of $F$ coincides with the Riemannian connection of $\tilde{a}$.   Assume that 	
	$(P, Y)$ be a flag in $T_{e}(G)$ such that $\{ U, Y\}$ is an orthonormal basis of $P$ with respect to $\left\langle \ , \ \right\rangle $. Then the flag curvature of the flag $(P, Y)$ is given by
	\begin{equation}{\label{flagexpeq9}}
	K(P, Y) = \dfrac{\bigg[  2\left\langle X, \big[ Y, \left[ U, Y\right]_{\mathfrak{m}} \big]_{\mathfrak{m}}\right\rangle    \left\langle X, U\right\rangle  +
			 \left\langle U, \big[ Y, \left[ U, Y\right]_{\mathfrak{m}} \big]_{\mathfrak{m}}\right\rangle   \Big( 1- \left\langle X, Y\right\rangle \Big)	\bigg] }{4 e^{2\left\langle  X, Y\right\rangle}  \bigg[ \Big( 1- \left\langle X, Y\right\rangle \Big) \left( 1 + 4 \left\langle X, Y\right\rangle^2 \right)   + \left\langle X, U\right\rangle^2 \left( 1+ 8 \left\langle X, Y\right\rangle^2  \right)  \bigg]  }.
	\end{equation}
\end{corollary}
\begin{proof}
	
	Since $ \left\langle \ , \ \right\rangle $ is bi-invariant, we have 
	$$ R(U, Y)Y= \dfrac{1}{4}\big[ Y, \left[ U, Y\right] \big].$$
	Substituting above value in equation (\ref{flagcurinfeq}), we get 
	\begin{equation*}
		K(P, Y) = \dfrac{\Bigg[  2 \left\langle X, \dfrac{1}{4}\big[ Y, \left[ U, Y\right]_{\mathfrak{m}} \big]_{\mathfrak{m}} \right\rangle \left\langle X, U \right\rangle+ \left\langle U, \dfrac{1}{4}\big[ Y, \left[ U, Y\right]_{\mathfrak{m}} \big]_{\mathfrak{m}} \right\rangle \Big( 1- \left\langle X, Y\right\rangle \Big) \Bigg]}{e^{2\left\langle  X, Y\right\rangle}  \bigg[ \Big( 1- \left\langle X, Y\right\rangle \Big) \left( 1 + 4 \left\langle X, Y\right\rangle^2 \right)   + \left\langle X, U\right\rangle^2 \left( 1+ 8 \left\langle X, Y\right\rangle^2  \right)  \bigg]}.
	\end{equation*}
	Simplifying the above equation, we get equation (\ref{flagexpeq9}).
\end{proof}

\end{document}